\documentclass[11pt]{amsart}
\usepackage{geometry}                
\usepackage{graphicx}
\usepackage{amssymb}
\usepackage{epstopdf}
\usepackage{tikz-cd, amscd} 
\usepackage[colorlinks=true, allcolors=blue]{hyperref}
\usepackage{url,xcolor}

\title[On Twisting Functions]{On twisting functions}
\author{Li Cai}
\address{Department of Pure Mathematics, Xi'an Jiaotong-Liverpool University, 111 Ren'ai Road, Dushu Lake Higher Education Town, Suzhou 215123, Jiangsu, China}
\email{Li.Cai@xjtlu.edu.cn}


\newtheorem{thm}{Theorem}[section]

\newtheorem{cor}[thm]{Corollary}
\newtheorem{lem}[thm]{Lemma}
\newtheorem{prop}[thm]{Proposition}

\theoremstyle{definition}

\theoremstyle{definition}

\theoremstyle{remark}

\def\co{\colon\thinspace}
\def\wt{\widetilde}
\newcommand\numberthis{\addtocounter{equation}{1}\tag{\theequation}}

\begin{document}
\begin{abstract}    
In this work, we unify different constructions of Kan's loop group $GX$ for a reduced simplicial set $X$ topologically, by identifying its geometric realization $|GX|$ as different submonoids of $\Omega|X|$, the monoid of based Moore loops on $|X|$. Then we construct a cubical subcomplex $|CX|\subset |GX|$ as a submonoid and prove that after inverting all elements of degree $0$ in $CX$, the inclusion $|S^{-1}CX|\subset |GX|$ is a (weak) homotopy equivalence. Our construction is functorial and explicit, without using inductions.  
\end{abstract}

\maketitle
\section{Introduction}
According to Dwyer and Kan \cite{DK84}, 
we have a pair of adjoint functors 
\[
     G\co \mathrm{sSets}\longleftrightarrow  \mathrm{sGroupoids} \co \overline{W}    
\]
between the homotopy theory of the category $\mathrm{sSets}$ of simplicial sets and the homotopy theory of the category $\mathrm{sGrpoids}$ of simplicial groupoids, giving an equivalence between them. This equivalence specializes to that between reduced simplicial sets and simplicial groups, by earlier works of Kan \cite{Kan58b} and May \cite{May68}. 

Let $X$ be a reduced simplicial set (i.e., it has a single element in dimension $0$). The definition of a simplicial group $GX$ depends on the choice of a twisting function $\tau$. In fact, such a choice is not unique, and we list all four choices in the literature (we only chose some early works, as it would be too long to list all of them):
 \begin{table}[htp]
\caption{The four versions of $\tau$}
\begin{center}
\begin{tabular}{|c|c|c|}
 Choices  & Face and Degeneracy Maps ($x\in X_n$) &  Used in Literature \\
            \hline
 a1)-b1)  & $d_0\tau x=(\tau d_1x) (\tau d_0 x)^{-1}, \   \tau s_0x=1;$ &  Curtis \cite{Cur71}\\
            &              $d_i\tau x=\tau d_{i+1}x, s_i\tau x=\tau s_{i+1}x, i=1,\ldots,n-1$ & \\
            \hline
a1)-b2) &  $d_0\tau x=(\tau d_0 x)^{-1}(\tau d_1x), \  \tau s_0x=1;$    & Gugenheim \cite{Gug60}, Szczarba \cite{Szc61}, May \cite{May68}  \\
             & $d_i\tau x=\tau d_{i+1}x, s_i\tau x=\tau s_{i+1}x,\ i=1,\ldots,n-1$ & \\
             \hline
a2)-b1)  & $d_{n-1}\tau x=(\tau d_n x)^{-1} (\tau d_{n-1}x),  \ \tau s_nx=1$; &   Berger \cite{Ber95}\\
            &              $d_i\tau x=\tau d_i x, s_i\tau x=\tau s_i x, i=0,\ldots, n-2$ & \\
            \hline
a2)-b2)  & $d_{n-1}\tau x= (\tau d_{n-1}x) (\tau d_n x)^{-1}, \  \tau s_nx=1$; &   Kan \cite{Kan58a} \\
            &              $d_i\tau x=\tau d_i x, s_i\tau x=\tau s_i x, i=0,\ldots, n-2$ & \\
\end{tabular}
\end{center}
\label{default}
\end{table}%

The first part of this work is to show that the relations between the choices of the twisting functions essentially come from two choices of inclusions $|GX|\subset \Omega|X|$ as a morphism of topological monoids, $\Omega|X|$ the Moore loops on the geometric realization $|X|$ of $X$. More precisely, we show that by considering the contractible total space $PX$ in fibration $PX\to X$ as a1) paths beginning with the base point, or a2) paths ending with it; together with the multiplication of two loops $\gamma,\gamma'$ is either b1) $\gamma \cdot\gamma'$ (as their concatenation) or b2) $\gamma'\cdot\gamma$ (as a composition of morphisms $\bullet\stackrel{\gamma}{\rightarrow} \bullet \stackrel{\gamma'}{\rightarrow} \bullet$), we get all four versions of twisting functions $\tau\co X\to GX$ listed in the table above.

 In the second part of this work we consider a further reduction from $GX$ to a cubical complex $CX$, by merging the simplices together (following Berger \cite[pp. 29--31]{Ber95}). More precisely, 
for an element $x=[0,\ldots,n+1]_x\in X_{n+1}$, $n\geq -1$, and let $g=(g_1,\ldots, g_n)\in S_n$ be a permutation $g\co \{1,\ldots, n\}\to \{1,\ldots, n\}$, $g_i=g(i)$ ($g=()\in S_0=\emptyset$ is a formal notation), we define ($T$ means triangulating)
\[cx=\begin{cases}\cup_{g\in S_n}Tcx(g) & n\geq 1\\
          \tau x & n=0\\
        
           \end{cases}
\]
where  
\begin{equation*}
Tcx(g)=\prod_{r=0}^{n}\tau[0,\alpha_{r1},\alpha_{r2},\ldots, \alpha_{rn},r+1]_x, \ \alpha_{rj}=
\max\left(\{0,1,\ldots, r\}\cap\{0,g_1,g_2,\ldots, g_j\}\right),  
\end{equation*}
$j=1, \ldots,n$, in which $[f_0,\ldots, f_{n+1}]_x$ is the face $X(f)(x)$ of $x$ associated with the non-decreasing function 
$f\co \{0,\ldots,n+1\}\to \{0,\ldots,n+1\}$, $f(i)=f_i$ (here we consider a simplicial set as a contravariant functor). We prove that the $n!$ simplices $Tcx(g)$ match well to give a desired cube; let $CX$ be a free monoid generated by cubes $cx$, $x$ running through all non-degenerate simplices of $X$, then the inclusion
\[|S^{-1}CX|\subset |GX|\] induces a (weak) homotopy equivalence as topological monoids, here  $S^{-1}CX$ is the localization of the free monoid $CX$ with respect to the collection $S$ of degree $0$ elements (i.e., all possible products of $\tau x$, with $\dim x=1$). 

 As a remark, a similar construction of cubical complexes in $GX$ appears in a recent work \cite{Fra25} by Franz (assuming that $X$ is $1$-reduced), realizing the categorical approach \cite{MRZ23} by Minichiello, Rivera and Zeinalian. Their construction is based on the original chain map given in Szczarba \cite{Szc61}. As a comparison, a simplex in a cube is of the form \[\pm (\tau x_1')^{-1}\ldots (\tau x_n')^{-1},\] with certain faces $x_1',\ldots, x_n'$ of $x$ that are inductively defined, using the Szczarba operators. The main advantage of our triangulation $Tcx(g)$ is that we can write it down immediately once $g\in S_n$ is given, without using inductions. Moreover, each element is a product of the generators $\tau x$ of $GX$, without their inverses $(\tau x)^{-1}$ involved. 
 
This work is topological in nature, without using the language of model categories. On the homology of loop spaces, for a commutative ring $R$, the differential graded algebra $(RCX,d)$ given by the free monoid $CX$ endowed with the differential $d$ by taking faces of cubes coincides with the original cobar complex constructed by Adams (when $X$ is $1$-reduced).  Furthermore, by Baues \cite{Bau98}, it admits a comultiplication which is coassociative (and cocommutative up to homotopy), so that its homology coincides with $H_*(\Omega|X|;R)$ as Hopf algebras. Explicit calculations and a more detailed comparison with the Szczarba operators will be given in a subsequent work.

 The paper is organized as follows. In Sections \ref{sec:1}, \ref{sec:2}, we review Berger's simplicial category $\Gamma X$, and show that the topological monoid $\Omega |X|$ of Moore loops on a reduced simplicial set $X$ is homotopy equivalent to $|\Omega X|$, where $\Omega X$ is a simplicial group consisting of morphisms in $\Gamma X$ whose source and target are the base point. All results in these two sections are due to Berger \cite{Ber95}. In Section \ref{sec:3} we use Berger's reduction from $\Omega X$ to its subgroup $GX$, to show that the difference between the four versions of twisting functions $\tau\co X\to GX$ in the literature comes from different morphisms $GX\to \Omega X$ of simplicial groups. In Section \ref{sec:4} we illustrate how simplices in $GX$ are merged into cubes, one cube for each non-degenerate $x\in X$, so that we have  a cubical complex $|CX|$ as a subcomplex of all singular cubes on $|GX|$, and the morphism $|T|\co |S^{-1}CX|\to |GX|$ of topological monoids induces a homotopy equivalence.

The author thanks Matthias Franz for pointing out Berger's work \cite{Ber95} on possible morphisms $|GX|\subset  \Omega|X|$, and for helpful discussions with him.  During this work he was supported in part by National Natural Science Foundation of China (grant no. 11801457).
 
\section{The Simplicial Category $\Gamma X$}\label{sec:1}
Let $X=(X_n)_{n=0}^\infty$ be a simplicial set. Throughout this work, we frequently identify an $n$-simplex $x_n\in X_n$ with the singular simplex
\[
        |x_n|: |\Delta^n|\to |X|
\]
on the geometric realization $|X|$,
which is induced from the simplicial map $x_n:\Delta^n\to X$ given by the Yoneda lemma, here $\Delta^n=(\Delta^n_k)_{k=0}^\infty$ is the simplicial set with $\Delta^n_k$ collecting all non-decreasing maps $[k]\to[n]$. As a singular simplex with barycentric coordinates $(t_0,\ldots, t_n)$, we have the face and degeneracy maps $d_i: X_n\to X_{n-1}$ and $s_i: X_n\to X_{n+1}$ given by 
\begin{align*}
      |d_ix_n|(t_0,\ldots, t_{n-1})&=|x_n|(t_0,\ldots, t_{i-1},0, t_{i},\ldots, t_{n-1}), \\ 
       |s_ix_n|(t_0,\ldots, t_{n+1})&=|x_n|(t_0,\ldots, t_{i-1},t_i+t_{i+1},\ldots, t_{n+1}), \numberthis \label{eq:o1}
\end{align*}
$i=0,\ldots,n$. Let $\Gamma X=(\Gamma_n X)_{n=0}^\infty$ be the simplicial object in which each $\Gamma_n X$ is a small category, $n\geq 0$, whose face and degeneracy maps $d_i,s_i$ are functors between them. More precisely, $Obj(\Gamma_n X)=X_n$ coincides with $X_n$, and $Mor(\Gamma_n X)$ is generated by 
\[X_{n+1}\times [n]\times\{-1,1\}=\{(x_{n+1},i),(x_{n+1},i)^{-1}\mid x_{n+1}\in X_{n+1}, i=0,\ldots, n\},\] 
in which an element is called an \emph{elementary prism},  where $x_n,x_n'\in X_n$ is connected by $(x_{n+1},i): x_n\to x_n'$ if and only if 
$d_{i+1}x_{n+1}=x_n$ and $d_{i}x_{n+1}=x_n'$, $i=0,\ldots,n$. As a singular $(n+1)$-simplex, $|(x_{n+1},i)|$ coincides with the Moore paths parametrized by $d_{i+1}x_{n+1}$, which is explicitly given by the formula
\begin{equation}
     |(x_{n+1},i)|(t,t_0,\ldots,t_n)=|x_{n+1}|(t_0,t_1,\cdots, t_{i-1}, t_i-t, t, t_{i+1},\cdots, t_{n})\label{def:xii}
\end{equation}
here $t\in [0,t_i]$, $t_i=1-\sum_{j\not=i}t_j$ depending on the coordinates $t_j$, $j\not=i$. 

Notice that for every $x_n\in X_n=Obj(\Gamma_n X)$, by \eqref{eq:o1}, 
\begin{equation*}
     |(s_ix_n,i)|(t,t_0,\ldots,t_n)=|s_nx_n|(t_0,t_1,\cdots, t_{i-1}, t_i-t, t, t_{i+1},\cdots, t_{n})=|x_n|(t_0,\ldots,t_n),
\end{equation*}
hence $(s_ix_n,i)=\mathrm{id}_{x_n}$, $i=0,\ldots,n$. 
We define the elementary prism $(x_{n+1},i)^{-1}: d_ix_{n+1}\to d_{i+1}x_{n+1}$ by 
\begin{equation}
     |(x_{n+1},i)^{-1}|(t,t_0,\ldots,t_n)=|x_{n+1}|(t_0,t_1,\cdots, t_{i-1},  t, t_{i}-t, t_{i+1},\cdots, t_{n}).\label{def:xiiii}
\end{equation}
A general element in $Mor(\Gamma_n X)$, called an $n$-\emph{prism}, is a composition of elementary ones, whenever possible. It is convenient to denote a composition 
\[(x_{n+1}',i')^{\varepsilon'}\circ (x_{n+1},i)^{\varepsilon}\in Mor(\Gamma_n X), \quad \varepsilon, \varepsilon'\in \{1,-1\},\]
of two elementary prisms as the product \[(x_{n+1},i)^{\varepsilon} (x_{n+1}',i')^{\varepsilon'},\] which is understood as the concatenation of two path segments 
\[ x_n\stackrel{(x_{n+1},i)^{\varepsilon}}{\longrightarrow} x_n' \stackrel{(x_{n+1}',i')^{\varepsilon'}}{\longrightarrow}  x_n'' \]
parametrized by elements in $X_n$. 
In this way we write a general $n$-prism $\gamma\in Mor(\Gamma_n X)$ in the form
\begin{equation}
 \gamma_s=\prod_{k=1}^l(\xi_k,i_k)^{\varepsilon_k}, \quad \varepsilon_k=\pm 1,  \label{def:nprism}
\end{equation}
in which the target of $(\xi_k,i_k)^{\varepsilon_k}$ is the source of $(\xi_{k+1},i_{k+1})^{\varepsilon_{k+1}}$, $k=1,\ldots,l-1$. 

The face and degeneracy functors $d_j:\Gamma_n X\to \Gamma_{n-1}X$, $s_j: \Gamma_n X\to \Gamma_{n+1}X$ are defined through their operations on the parameters, namely 
\begin{align*}
|d_j(x_{n+1},i)|(t,t_0,\ldots,t_{n-1})&=|(x_{n+1},i)|(t,t_0,\ldots,t_{j-1},0,t_{j+1},\ldots, t_{n-1})\\ 
|s_j(x_{n+1},i)|(t,t_0,\ldots,t_{n+1})&=|(x_{n+1},i)|(t,t_0,\ldots,t_{j-1},t_{j}+t_{j+1},t_{j+2},\ldots, t_{n+1}). 
\end{align*}
It can be checked straightforwardly that we have commutative diagrams 
\begin{equation}
    \begin{CD}
   x_n@>(x_{n+1},i) >> x_n'\\
   @Vd_j VV                         @Vd_j VV\\
   d_jx_n@>d_j(x_{n+1},i)>> d_j x_n',
\end{CD} \quad\quad
\begin{CD}
   x_n@> (x_{n+1},i) >> x_n'\\
   @Vs_j VV                         @Vs_j VV\\
   s_jx_n@>s_j(x_{n+1},i)>> s_j x_n',
\end{CD}
\label{def:CD1}
\end{equation}
$j=0,\ldots,n$, hence the face and degeneracy functors satisfy necessary axioms so that $\Gamma X$ is a simplicial object. Moreover, the collection $Mor(\Gamma X)=(Mor(\Gamma_n X))_{n=0}^{\infty}$ of prisms is closed under face and degeneracy maps, thus it is a simplicial set itself. 

\section{Reduction of Moore loops to simplicial prisms} \label{sec:2}
Although we have defined $(x_{n+1},i)$, $(x_{n+1},i)^{-1}$ (see \eqref{def:xii}, \eqref{def:xiiii}), the relations
\begin{equation} 
(x_{n+1},i)(x_{n+1},i)^{-1}=\mathrm{id}_{d_{i+1}x_{n+1}}, \quad (x_{n+1},i)^{-1}(x_{n+1},i)=\mathrm{id}_{d_ix_{n+1}}\label{rel:xi}
\end{equation}
hold only after homotopy. We define $h\Gamma X$ as the simplicial groupoid obtained from $\Gamma X$ by adding relations \eqref{rel:xi}, so that every morphism is invertible. 

Let $|X|^{[0,+\infty)}$ be the collection of Moore paths, namely the continuous maps $\gamma: [0,\infty)\to |X|$ such that $\gamma(t)$ is constant whenever $t$ is sufficiently large. We see that the function $\rho\co |X|^{[0,+\infty)}\to [0,+\infty)$ given by the infimum $\rho(\gamma)=\inf\{t\mid \gamma(t+t')=\gamma(t),\forall t'\geq 0\}$ is well-defined. For  $\gamma,\gamma'\in |X|^{[0,+\infty)}$,  their concatenation is given by
\[
\gamma\cdot\gamma'(t)=\begin{cases}\gamma(t) & t\in[0,\rho(\gamma)]\\
\gamma'(t-p) & t\in [\rho(\gamma),+\infty)
\end{cases}
\]
when $\gamma(\rho(\gamma))=\gamma'(0)$. It is well-known that equipped with the compact open topology and concatenations of paths, $|X|^{[0,+\infty)}$ is a topological monoid which is strictly associative, moreover, the automorphism of $|X|^{[0,+\infty)}$ sending $\gamma$ to $\gamma^{-1}$ is continuous, where  $\gamma^{-1}(t)=\gamma(\rho(\gamma)-t)$ for $t\in [0,\rho(\gamma)]$ and $\gamma^{-1}(t)=\gamma(0)$ for $t\in [\rho(\gamma),+\infty)$.

Let $|Mor(\Gamma X)|$ be the geometric realization of the simplicial set $Mor(\Gamma X)$ of prisms. We define a morphism
\begin{equation}
\varrho:|Mor(\Gamma X)|\to |X|^{[0,+\infty)} \label{def:f1}
\end{equation}
 of topological monoids, such that an $n$-simplex $|\gamma_s|\co |\Delta^n|\to |Mor(\Gamma X)|$ associated to the prism $\gamma_s$ of the form \eqref{def:nprism} is sent to 
\begin{align*}
&\varrho(|\gamma_s|)(t,t_0,\cdots, t_n)=\\
&\begin{cases}
 |\xi_k|(t_0,t_1,\cdots, t_{i_k-1}, t_{i_k}-(t-\sum_{s<k}t_{i_s}), t-\sum_{s<k}t_{i_s}, t_{i_k+1},\cdots, t_{n}) & t\in [\sum_{s<k}t_{i_s},\sum_{s\leq k}t_{i_s}]\\
 |(\xi_l,i_l)|(t_{i_l},t_0,\cdots,t_n) & t\geq t_{i_1}+\ldots+t_{i_l},
\end{cases}
\end{align*} 
as the concatenation of parametrized paths $|(\xi_k,i_k)|$ (see \eqref{def:xii}), $k=1,\ldots,l$.  It is easy to check that $\varrho$ is continuous and preserves the multiplications, since the image under $\varrho$ of the concatenation $\gamma_s\cdot\gamma'_s$ of two prisms is the concatenation $\varrho(\gamma_s)\cdot \varrho(\gamma_s')$ of parametrized Moore paths.

To simplify $\varrho$ by homotopy, we consider the quotient $|X|^{[0,+\infty)}\to h|X|^{[0,+\infty)} $ by adding the relations $\gamma\cdot\gamma^{-1}=\mathrm{id}_{\gamma(0)}$ and $\gamma^{-1}\cdot\gamma=\mathrm{id}_{\gamma^{-1}(0)}$ for all $\gamma\in |X|^{[0,+\infty)}$. The composition of this quotient with $f$ gives a map $|Mor(h\Gamma X)|\to h|X|^{[0,+\infty)}$, thus by \eqref{rel:xi} we get a well defined map 
\[
              h\varrho\co |Mor(h\Gamma X)|\to h|X|^{[0,+\infty)}.
\]   
Notice that by definition, we have $Obj(h\Gamma X)=Obj(\Gamma X)=X$ as simplicial sets. Moreover, by taking the source and target simplices connected by a prism $\gamma_s\in Mor(\Gamma_hX)$, we have a simplicial map 
\[
\begin{CD}
 Mor(h\Gamma X)@>(\pi_s,\pi_t)>> X \times X,
 \end{CD}
\]
whose geometric realization satisfies the commutative diagram
\[
\begin{CD}
Mor(h\Gamma X)@>(|\pi_s|,|\pi_t|)>> |X| \times |X|\\
@V h\varrho VV                                                      @V=VV\\
h|X|^{[0,+\infty)}@>(|\pi_s|,|\pi_t|)>> |X| \times |X|,
\end{CD}
\]
where $|\pi_s|,|\pi_t|$ are projections onto the source and target of Moore paths, respectively. We recall the following result of Berger, which we shall not prove here (see \cite[Prop. 1.3, page 6]{Ber95}).
\begin{thm}\label{thm:berger1}
The simplicial map $(\pi_s,\pi_t)\co Mor(h\Gamma X)\to X\times X$ is a Kan fibration for any pathwise connected simplicial set $X$ ($X$ itself is not necessarily Kan).  The geometric realizations $|f|,|g|$ of two given simplicial maps $f,g: Y\to X$ are homotopic, if there exists a lifting $H\co Y\to Mor(h\Gamma X)$ such that $\pi_s\circ H=f$ and $\pi_t\circ H=g$.
\end{thm}
The homotopy in the theorem above is called \emph{prismatique} in \cite{Ber95}.  Now let $v\in X_0$ be a given vertex, where the same notation shall be used for its image $v\in X_n$ under degeneracy maps, $n\geq 1$. Consider the simplicial subset $PX_v\subset Mor(h\Gamma X)$ whose $n$-simplices ${P}_nX_v$ is the collection of $\gamma_s\in Mor(h\Gamma_nX)$ whose target is $v\in X_n$, i.e., $\pi_t(\gamma_s)=v$, in which $\Omega X_v\subset PX_v$ is simplicial subset consisting of prisms with source $v$. Similarly we denote the Moore paths and loops with target $v$ by $P|X_v|, \Omega |X_v|\subset h|X|^{[0,+\infty)}$, respectively.  Clearly  $\Omega X_v$ is a simplicial group acting on $PX_v$ from the right. As a conclusion, we have a commutative diagram
\begin{equation}
\begin{CD}
            |\Omega X_v| @>>> |P X_v| @>|\pi_s|>> |X| \\
               @V h\varrho VV                  @Vh\varrho VV             @V=VV\\
            \Omega |X_v| @>>> P |X_v| @>|\pi_s|>> |X|.
\end{CD}\label{diag:Omega}
\end{equation}
\begin{thm}\label{thm:main5}
The space $|PX_v|$ is contractible. Consequently the morphism \[h\varrho\co|\Omega X_v|\to  \Omega |X_v| \] of topological groups induces a weak homotopy equivalence.
\end{thm}
\begin{proof}
The first row of \eqref{diag:Omega} is a Serre fibration since it is induced from the Kan fibration $(\pi_s,\pi_t)$, by Theorem \ref{thm:berger1}. 
Together with the well-known result that the second row is a Serre fibration, $|\Omega X_v|\to  \Omega |X_v|$ is a weak homotopy equivalence, once we show that all homotopy groups of $|PX_v|$ vanish.  Let $\zeta\co S^n\to |P X_v|$ be a continuous map representing an element in the homotopy group. By simplicial approximation, up to homotopy $\zeta$ is represented by a simplicial map from a triangulation of $S^n$ to $PX_v$.

Berger \cite[Lemma 2.5, pp. 19--20]{Ber95} proved that there exists a section $s_X\co PX_v\to PPX_v$, so that $s_X \zeta\in PPX_v$ with $\pi_ss_X\zeta=\zeta$ and $\pi_ts_X\zeta=\mathrm{id}_v$. By Theorem \ref{thm:berger1}, it gives the desired homotopy from $|\zeta|$ to $|\mathrm{id}_v|$. 
\end{proof}

\section{Kan's groups and twisting functions}\label{sec:3}
Let $G$ be a simplicial group acting on a simplicial set $Z$ from the left. Recall that in a twisted Cartesian product $X\times_{\tau} Y$ with a twisting function $\tau: X\to G$ of degree $-1$, for $(x,y)\in X_n\times Z_n$, $n\geq 1$, we have $d_i(x,z)=(d_ix,d_iz)$ (resp. $s_i(x,z)=(s_ix,s_iz)$) for $i=0,\ldots, n-1$ (resp. $i=1,\ldots,n$) and $d_n(x,y)=(d_nx, \tau x\cdot d_nz)$, which satisfies $d_{n-1}\tau x=(\tau d_n x)^{-1}(\tau d_{n-1}x)$ and $\tau s_n x=\mathrm{1}_n$. When $Z=G$, $G$ acts on $X\times_\tau G$ on which $G$ from the right in an obvious way. 

A \emph{pseudosection} $\iota\co X\to E$ of an epimorphism $\pi\co E\to X$ is a map such that the composition $\pi\iota=\mathrm{id}_x$, where on face and degeneracy maps we have $\iota s_i=s_i\iota$ for all $i$ and $d_i\iota=\iota d_i$ for $0\leq i<n$, while $d_n\iota x=\iota d_nx\cdot\tau x $ for all $x\in X_n$.
\begin{thm}[\cite{Ber95}]\label{thm:main3}
Let $X$ be a reduced simplicial set with base point $v$. The path fibration $\pi_s\co P X_v\to X$ admits a pseudosection $\iota\co X\to PX_v$ such that $PX_v$ is isomorphic to $X\times_\tau \Omega X_v$ with  $\tau x=(\iota d_nx)^{-1} (x,n-1) (\iota d_{n-1}x)$ for all $x\in X_n$, $n\geq 1$.

The simplicial subgroup $GX=(G_nX)_{n=0}^{\infty}\subset \Omega X_v$, in which every group $G_nX$ is generated by elements $\tau x$, as $x$ runs through $X_{n+1}$ which is not of the form $x=s_ny$ for some $y\in X_n$, is a free group with these generators.
\end{thm}
We begin with the construction of the pseudosection $\iota$. Henceforth, we represent an element $x\in X_n$ by $[0,1,\cdots,n]_x$, the image of the non-degenerate $n$-simplex of $x: \Delta^n\to X$ by the Yoneda Lemma. In this way 
\[s_ix=[0,1,\cdots, i-1, i,i, i+1,\cdots, n]_x,\] 
as a tuple with $n+1$ entries, and 
\[d_ix=[0,1,\cdots, i-1,i+1,\cdots,n]_x\] being a tuple with $n-1$ entries, for $0\leq i\leq n$; the elementary prism $(\xi,i)$ (see \eqref{def:xii}) for $\xi=[0,1,\cdots,n+1]_{\xi}$ induces the morphism
\[
\begin{CD}
[0,1,\cdots, i-1,i,i+2,\cdots,n+1]_{\xi}@>(\xi,i)>> [0,1,\cdots, i-1,i+1,i+2,\cdots,n+1]_{\xi}.
\end{CD}
\]
We connect $x\in X_n$ to the base point by a prism $\iota x\in Mor(h\Gamma_n X)$,
\begin{align}
\iota x=&[0,\cdots, n-2,n-1,n]_x  \to [0,\cdots, n-2,n,n]_x\to   \label{def:jx}\\
    &[0,\cdots, n-3,n,n,n]_x\to \cdots \to [0,n,\cdots,n]_x \to [n,\cdots,n]_x, \nonumber
\end{align}
where the last $n$-simplex is $v$ since $X$ is reduced. It can be explicitly written it down as a concatenation of elementary prisms
\begin{align*}
&(s_nx,n-1)(s_{n}s_{n-1}d_{n-1}x,n-2)\cdots (s_{n}\cdots s_{n-k+2}s_{n-k+1}d_{n-k+1}\cdots d_{n-2}d_{n-1}x,n-k) \\
&\cdots (s_{n}\cdots s_{1}d_1\cdots d_{n-1}x,0).
\end{align*}

Now we define 
\begin{align}
\tau x&=(\iota d_nx)^{-1}(x,n-1) (\iota d_{n-1}x) \nonumber \\
&=[n-1,\cdots,n-1]_x\to [0,n-1,\cdots,n-1]_x\to \cdots\to [0,1,\cdots,n-2,n-1]_x \label{def:taux}\\
&\to [0,1,\cdots,n-2,n]_x\to[0,1,\cdots,n-3,n,n]_x\to\cdots\to[0,n,\cdots,n]_x\to [n,n,\cdots,n]_x \nonumber
\end{align}
as an element in $\Omega_{n-1}X_v$. 

It remains to check the formulas involving the face and degeneracy maps. The relations $d_i\iota =\iota d_i$ are clear for $0\leq i\leq n-1$, by \eqref{def:CD1}, \eqref{def:jx}. On the other hand, a comparison of 
\[
d_n\iota x=[0,\cdots, n-2,n-1]_x \to [0,\cdots, n-2,n]_x\to [0,\cdots, n-3,n,n]_x \to \cdots \to [n,n,\cdots,n]_x
\]
and
\[
\iota d_nx=[0,\cdots, n-2,n-1]_x \to [0,\cdots, n-3,n-1,n-1]_x \to \cdots \to [n-1,n-1,\cdots,n-1]_x
\]
shows that $d_n\iota x=(\iota d_nx)(\tau x)$, as desired. Similarly we verify $s_i\iota x=\iota s_ix$ for all $i$. On $\tau x$, for $i<n-1$ we have already proved $d_i\iota =\iota d_i$, then
\begin{align*}
d_i\tau x&=d_i\left((\iota d_nx)^{-1}(x,n-1)(\iota d_{n-1}x)\right)=(d_i\iota d_nx)^{-1}d_i(x,n-1)(d_i\iota d_{n-1}x)\\
&=(\iota d_id_nx)^{-1}(d_ix,n-2)(\iota d_id_{n-1}x)=
(\iota d_{n-1}d_ix)^{-1}(d_ix,n-2)(\iota d_{n-2}d_ix)=\tau d_ix,
\end{align*}
in which $d_i(x,n-1)=(d_ix,n-2)$ since they both give 
\[
[0,\ldots,i-1,i+1,\ldots,n-1]_x\to [0,\ldots,i-1,i+1,\ldots,n]_x,
\] and  
\begin{align*}
d_{n-1}\tau x&=d_{n-1}\left((\iota d_nx)^{-1}(x,n-1)\iota d_{n-1}x\right)=(d_{n-1}\iota d_nx)^{-1}\mathrm{id}_{d_{n-1}d_nx}(d_{n-1}\iota d_{n-1}x)\\
&=
((\iota d_{n-1}d_nx) (\tau d_nx))^{-1}(\iota d_{n-1}d_{n-1}x)(\tau d_{n-1}x)=(\tau d_nx)^{-1}(\tau d_{n-1}x)
\end{align*}
in which we have used $d_{n-1}\iota x'=(\iota d_{n-1}x')(\tau x')$ for $x'\in X_{n-1}$ and 
\[d_{n-1}(x,n-1)=[0,\ldots,n-2]_x\to [0,\ldots,n-2]_x=\mathrm{id}_{d_{n-1}d_nx}.\] 
 Finally, from \eqref{def:taux} we have 
$\tau s_n x=1_n$ since $s_nx=[0,1,\ldots,n-1,n,n]_x$; the relations $s_i\tau =\tau s_i$ with $0\leq i\leq n-1$ hold from
\[
s_i\tau x=(\iota s_id_nx)^{-1} (s_ix, n) (\iota s_i d_{n-1}x)=(\iota d_{n+1} s_{i}x)^{-1}(s_ix, n) (\iota  d_{n} s_ix)=\tau s_ix.
\]
\begin{proof}[Proof of Theorem \ref{thm:main3}]
The pseudosection $\iota\co X\to PX_v$ defined by \eqref{def:jx} gives a simplicial map $f\co X\times_\tau \Omega X_v\to P X_v$ that sends $(x,g)$ to $\iota x\cdot g $, preserving the $\Omega X_v$ actions from the right.  On the other hand, we construct $f': P X_v\to X\times_\tau \Omega X_v$ by sending an  $n$-prism $\gamma_s$ with source $\pi_s(\gamma_s)$ and target $\pi_t(\gamma_s)=v$,  to $(\pi_s(\gamma_s), (\iota\pi_s(\gamma_s))^{-1}\gamma_s)$.  It is easily checked that $f' (\iota x\cdot g)=(x, g)$ and $f(f'(\gamma_s))=\gamma_s$, therefore $f'$ is the inverse of $f$, as morphisms of simplicial sets.

The second statement holds by checking the elementary prisms that appear in the definition of $\tau x$ (see \eqref{def:taux}, in which the key step is $[0,\ldots, n-2,n-1]_x\to [0,\ldots, n-2,n]_x$, and the definition of $h\Gamma X$ that all relations come essentially from \eqref{rel:xi} and $(s_{n}y,n)=1$, $y\in X_n$. 
\end{proof}
As a conclusion, we have a well-defined simplicial map $X\times_\tau GX\to X\times_\tau \Omega X_v$ induced by the inclusion $GX\to \Omega X_v$ and the identity on $X$, which is a morphism of twisted cartesian products over $X$. The contractibility of $PX_v$, as well as that of $X\times_\tau GX$ which is a well-known theorem of Kan \cite{Kan58a}, imply the following, after a comparison of the long exact sequences of homotopy groups.
\begin{cor}\label{thm:Berger}
The inclusion $GX\to \Omega X_v$ of simplicial groups induces a weak homotopy equivalence.
\end{cor}   



Originally, a twisting function $\tau$ defined by Kan \cite[Page 293]{Kan58a} satisfies 
\begin{equation} d_{n-1}\tau x= (\tau d_{n-1}x) (\tau d_n x)^{-1}.\label{def:Kan}\end{equation}
This can be done by considering concatenations of Moore paths as compositions of morphisms, more precisely, the concatenation $a\cdot b$ of Moore paths becomes 
\[b\circ a=\bullet\stackrel{a}{\rightarrow} \bullet \stackrel{b}{\rightarrow} \bullet.\]  In this way we reverse the order of all possible products elementary prisms, resulting in \eqref{def:Kan}.

Alternatively, one could define $PX_v\subset h\Gamma X$ as the prisms beginning at the vertex $v$, on which $\Omega X_v$ acts from the left. In this case the pseudosection $\iota: X\to PX_v$ has the form
\begin{align}
\iota x=&[0,\cdots,0]_x\to [0,0,\cdots,0,n]_x\to [0,\cdots,0,n-1,n]\to\cdots\to \\
     & [0,0,0,3,\cdots, n]_x
\to [0,0,2,\cdots, n-1,n]_x  \to [0,1,2,\cdots, n-1,n]_x, \nonumber
\end{align}
namely 
\[
  (s_{n-1}\ldots s_1s_0d_1\ldots d_{n-1} x,n)  \ldots (s_ks_{k-1}\ldots s_1s_0d_1d_2\ldots d_k x,k+1)\ldots (s_1s_0d_1x,2)(s_0x,1)
\]
as a product of elementary prisms.
 Similarly, we define
\begin{align*}
\tau x=&(\iota d_1x)(x,0)(\iota d_{0}x)^{-1}=[0,\cdots,0]_x\to[0,\cdots,0,n]_x\to\cdots\to [0,2,\cdots,n]_x\\
 &\to [1,2,\cdots, n]_x\to [1,1,3,\cdots,n]_x\to\cdots\to [1,\cdots,1]_x,
\end{align*}
implying that $\tau s_0x=1_n\in \Omega X_v$ since $s_0x=[0,0,1,\ldots,n]_x$. It can be checked straightforwardly that $s_i\iota =\iota s_i$ for all $i$ and $d_i\iota =\iota d_i$ for $1\leq i\leq n$.  On $d_0$ we have
\begin{align*}
d_0\iota x&=[0,\cdots,0]_x\to [0,\cdots,0, n]_x\to\cdots\to [0,2,\cdots,n]_x\to [1,2,\cdots,n]_x\\
\iota d_0x&=[1,\cdots,1]_x\to[1,\cdots,1, n]_x\to\cdots\to[1,1,3,\cdots,n]_x\to [1,2,\cdots,n]_x,
\end{align*}
whence the relation $d_0\iota x=(\tau x)(\iota d_0x)$. We omit the parallel verification of the relations $d_i\tau =\tau d_{i+1},$ for $i=1,\ldots,n-1$ and $d_0\tau x=(\tau d_1 x) (\tau d_0 x)^{-1}$, as well as $s_{i}\tau x=\tau s_{i+1}x$ for $i\geq 1$. 

The axioms of face and degeneracy maps on $\tau$ coincide with those in Curtis \cite[Page 133]{Cur71}, which enable us to define a twisting $\tau\co X\to G$ for a simplicial group $G$, and fiber bundles of the form $Y\times_\tau X$ as a twisted cartesian product, $G$ acting on the simplicial set $Y$ from the right, in which $d_0(y,x)=(d_0y \cdot \tau x, d_0x)$ for all elements $(y,x)$ of dimension $\geq 1$. As another version of Theorem \ref{thm:main3}, we have the identification 
\[
   PX_v= \Omega X_v\times_{\tau} X.
\]

Again, using compositions of morphisms instead of concatenations of Moore paths, the order of  elementary prisms in $\iota x$ and $\tau x$ shall be inverted, and we get $d_0\tau x=(\tau d_0 x)^{-1}( \tau d_1x)$.



\section{Singular Cubes in $GX$}\label{sec:4}
The geometric realization $|GX|$ of the simplicial group is a topological group, whose multiplication $|GX|\times |GX|\to |GX|$ is understood through local charts. More precisely, let $\sigma\in G_pX, \sigma'\in G_{p'}X$ be two simplices represented by simplicial maps $\sigma\co\Delta^p\to GX$, $\sigma'\co\Delta^{p'}\to GX$. Then the composition
\begin{equation}
\begin{CD}
 |\Delta^p|\times |\Delta^{p'}| @>|(\sigma,\sigma')|>> |GX|\times |GX| @>|m |>> |GX|\label{def:monoid1}
 \end{CD}
\end{equation}
gives the monoid structure of $|GX|$, where $m$ is the multiplication of $GX$. 

Notice that by \eqref{def:monoid1} we have a singular cell on $|GX|$ which is not simplicial. A further understanding can be done by the shuffle products.
Let $a,b$ be two letters. We define $\mathrm{Shuf}(a^p,b^q)$, $p,q$ two non-negative integers, as the set of all words of the form 
\begin{equation}
w=x_{p+q}\ldots x_{1}\label{def:w}
\end{equation} 
of length $p+q$, in which exactly $p$ letters are $a$ and $q$ letters are $b$ ($w$ is an empty word if both $p=q=0$). For such a word $w\in \mathrm{Shuf}(a^p,b^q)$, let $\mathrm{sign}(w)$ be the number of permutations of letters to change it into the word $b^qa^p$, which is well-defined $\mathrm{mod}$ $2$.  For instance, $\mathrm{Shuf}(a^2, b)=\{aab, aba, baa\}$ with $\mathrm{sign}(aba)=1$. 

Let $x_b(w)=x_{b_q}\ldots x_{b_{1}}$ (resp. $x_a(w)=x_{a_p}\ldots x_{a_{1}}$) be the subword of $w$ obtained by deleting all letters $x_i$ which are $a$ in $w$ (resp. by deleting all letters $x_i$ which are $b$ in $w$), and let $s_{x_b^-(w)}=s_{b_{q}-1}s_{b_{q-1}-1}\cdots s_{b_1-1}$ (resp. $s_{x_a^-(w)}=s_{a_{p}-1}s_{a_{p-1}-1}\cdots s_{a_1-1}$) be the corresponding iterated degeneracy operators. As an example, if $w=x_3x_2x_1$ with $x_1=x_2=a$ and $x_3=b$, we have $x_b(w)=x_3$, $x_a(w)=x_2x_1$, and $s_{x_b^-(w)}=s_2$, $s_{x_a^-(w)}=s_1s_0$, respectively. Let $GX\to RGX$ be the embedding into the corresponding simplicial group ring with coefficients from a commutative ring $R$. We define the product $\cdot\co RGX\otimes RGX\to RGX$ as the composition 
\begin{equation}
\begin{CD}
RGX\otimes RGX@>\nabla>> R(GX\times GX)@> Rm >> RGX
\end{CD}\label{def:cdot}
\end{equation}
where $RGX\otimes RGX$ is the tensor of graded $R$-modules over $R$, $Rm$ the multiplication induced by that of $GX$ and $\nabla$ the morphism of graded $R$-modules given by
\begin{equation}
  \nabla(\sigma\otimes\sigma')= \sum_{w\in \mathrm{Shuf}(a^p,b^{p'})}(-1)^{\mathrm{sign}(w)} (s_{x_b^-(w)}\sigma, s_{x_a^-(w)}\sigma'),\label{def:nabla}
\end{equation}
$\sigma\in RG_p,\sigma'\in RG_{p'}$, respectively, so that $\nabla(\sigma\otimes\sigma')$ is homogeneous of degree $p+p'$. To avoid confusion, in what follows the multiplication of two elements $g,g'\in G_n$ in the simplicial group $G$ shall be denoted by $gg'$.
It is convenient to use the notation
\begin{equation}
  \begin{array}{cccccccc}
     & x_1 & \ldots &x_i & x_{i+1} &\ldots & x_{p+p'} & \\
  \hline
     &n_0&\ldots &n_{i-1} & n_i &\ldots & n_{p+p'-1} & n_{p+p'}\\ 
     &n_0'&\ldots &n_{i-1} '& n_i' &\ldots & n_{p+p'-1}' & n_{p+p'}'\\ 
  \end{array}\label{def:w1}
\end{equation}
where 
\begin{align}
s_{x_b^-(w)}\sigma&=[n_0,\ldots,n_{p+p'}]_{\sigma}, \ n_0=0, \ n_{i}=\begin{cases} 
n_{i-1}+1 & x_i=a\\
n_{i-1}     & x_i=b,
\end{cases} \\
s_{x_a^-(w)}\sigma'&=[n_0',\ldots,n_{p+p'}']_{\sigma'}, \ n_0'=0, \ n_{i}'=\begin{cases} 
n_{i-1}' & x_i=a\\
n_{i-1}'+1     & x_i=b.
\end{cases}
\end{align}

\subsection{Cubes as unions of simplices in $GX$}
We say that a $CW$ complex $Y$ is \emph{cubical} if any of its characteristic maps, i.e., the attaching of an $n$-cell $e_n$, is represented by a continuous map $\square^n\to |Y|$ from the standard $n$-cube $e_n\co \square^n=[0,1]^n$ onto its image in $|Y|$, whose $i$-th top and bottom faces $d_i^0$ and $d_i^1$, respectively, $i=1,\ldots,n$, are attached by singular cubes $|d_i^te_n|\co\square^{n-1}\to |Y|$, with
\[
  |d_i^te_n|(t_1,\ldots, t_{n-1})=\begin{cases}
  |e_n|(t_1,\ldots, t_{i-1},0, t_{i}, \ldots, t_{n-1}) & t=1\\
  |e_n|(t_1,\ldots, t_{i-1},1, t_{i}, \ldots, t_{n-1}) & t=0,
  \end{cases}
  \ t_i\in[0,1],\ i=1,\ldots, n-1.
\] 

\begin{thm}\label{thm:main0}
Let $X$ be a reduced simplicial set. There exists a cubical complex $|CX|$ as a subcomplex of all singular cubes on $|GX|$, together with  a continuous map $|T|\co |CX|\to |GX|$, which satisfies the following properties. \begin{enumerate}
\item [a)] Every $x\in X_{n+1}$ gives a singular cube $|Tcx|\co \square^n\to |GX|$, whose image is a union of $n!$ simplices in $|GX|$ as its triangulation. If $x$ is degenerate, so is every simplex in this union.
\item [b)] As a topological monoid, $CX$ is generated by $cx$, by identifying $cx$ with the singular cube $|Tcx|$, as $x$ runs through all non-degenerate simplices of dimension $\geq 1$. Moreover, $CX$ is a free monoid with this set of generators.
\item [c)] On the $k$-th top and bottom faces $|d_k^0Tcx|,|d_k^1Tcx|\co \square^{n-1}\to |GX|$, $k=1,\ldots,n$, we have   
\[|d_k^0Tcx|=|Tc[0,\ldots,k]_x|\times |Tc[k,\ldots, n+1]_x|\circ (\pi_1,\pi_2),\quad |d_k^1Tcx|=|Tcd_kx|,
\] 
in which $|Tc[0,\ldots,k]_x|\times |Tc[k,\ldots, n+1]_x|\co \square^{k-1}\times\square^{n-k}\to |GX|$ is the cartesian product of singular cubes $|Tc[0,\ldots,k]_x|,|Tc[k,\ldots, n+1]_x|$, here $(\pi_1,\pi_2)\co\square^{n-1}\to \square^{k-1}\times\square^{n-k}$ is the projection onto its first $k-1$ and last $n-k$ coordinates, respectively.  In other words, in the topological monoid $CX$ we have 
\[
d_k^0cx=c[0,\ldots,k]_x \cdot c[k,\ldots, n+1]_x,\quad d_k^1cx=cd_kx,
\quad k=1,\ldots,n.
\]
\item [d)] The map $|T|$ is functorial with respect to simplicial maps. More precisely, a simplicial map $f: X\to X'$ between reduced simplicial sets $X,X'$ induces a morphism $Cf\co CX\to CX'$ of topological monoids, such that the diagram 
\[
\begin{CD}
                 |CX|@> |T|>> |G X|\\
                 @V|Cf|VV    @V |G f|VV \\
                 |CX'| @>|T|>>| GX'|
\end{CD}
\]
commutes.
\item [e)] The map $|T|\co |CX|\to |GX|$ induces a weak homotopy equivalence when $|X|$ is simply connected. 
\end{enumerate}
\end{thm}
As a remark, we describe the monoid structure of $CX$ in more detail. A general $n$-cube in $CX$ is a product  
\begin{equation}
    c(x_1,\ldots, x_l)=cx_1\cdot cx_2\cdot\ldots \cdot cx_l, \quad  \label{def:cx}
\end{equation}
of cubes $cx_1,\ldots, cx_l$ associated to simplices $x_1,\ldots, x_l\in X$ of dimensions $n_1+1,\ldots, n_l+1$, respectively, $\sum_{k=1}^ln_k=n$. Topologically it corresponds to the singular cube $|Tcx_1|\times |Tcx_2|\times \ldots \times |Tcx_l|$, which is well defined since $X$ is reduced, with 
$[n_{k}+1]_{x_{k}}=[0]_{x_{k+1}}$ for $k=1,\ldots, l-1$.
On the top and bottom faces we have  
\begin{equation}
 d_i^tc(x_1,\cdots,x_l)=
                          cx_1\cdot cx_2\cdot \ldots\cdot \left(d_{i-\sum_{k<j}n_k}^t cx_j\right)\cdot\ldots\cdot cx_l, \quad t=0,1,\label{def:face}
\end{equation}
where $j$ is the unique number such that $\sum_{k<j}n_k<i\leq \sum_{k\leq j}n_k$. The homotopy equivalence $|T|\co |CX|\to |GX|$, together with Corollary \ref{thm:Berger} and the well-known fact that we can ignore the degenerate cubes when passing to homology, we obtain the following classic theorem of Adams.
\begin{cor}[\cite{Ada56}]\label{cor:Adams}
   Let $(RCX,d)$ be the differential graded algebra associated to a reduced simplicial set $X$, which is freely generated (as an associative algebra) by $cx$ of degree $n$, as $x=[0,\ldots,n+1]_x$ runs through non-degenerate simplices of dimension $n+1\geq 1$, whose differential satisfies ($dcx=0$ if $n=0$)
\begin{equation}
    dcx=\sum_{i=1}^{n}(-1)^{i}(d_i^0cx-d_i^{1}cx)=\sum_{i=1}^{n}(-1)^i(c[0,\ldots,i]_x \cdot c[i,\ldots,n+1]_x-cd_ix)\label{eq:Ad1}
\end{equation}
where $cd_ix$ (resp. $c[0,\ldots,i]_x$ or $c[i,\ldots,n+1]_x$) vanishes whenever $d_ix$ (resp. $[0,\ldots,i]_x$ or $[i,\ldots,n(x)+1]_x$) is degenerate, and satisfies $d(cx\cdot cx')=d(cx)\cdot cx'+(-1)^{n}cx\cdot d(cx')$ for all generators $cx,cx'$. Then we have an isomorphism 
\[
     H_*(\Omega|X|;R)=H_*(RCX,d)
\]
of graded algebras when $|X|$ is simply connected.
\end{cor}
After a change of basis by $c'x=\begin{cases}
   cx & \dim x>1 \\
   cx-1 & \dim x=1 \\
   0 & \dim x=0,
\end{cases}$
 $(RCX,d)$ is freely generated by $c'x$, where
 the differential \eqref{eq:Ad1} becomes
\begin{equation}
    dc'x=-\sum_{i=0}^{n+1}(-1)^ic'd_ix+\sum_{i=0}^{n+1}(-1)^i c'[0,\ldots,i]_x \cdot c'[i,\ldots,n+1]_x.\label{eq:Ad2}
\end{equation}
\subsection{Triangulations of cubes and symmetric groups}
Let $x=[0,\ldots,n+1]_x$ be an $(n+1)$-simplex, $n\geq 0$, and let $g=(g_1,\ldots, g_n)\in S_n$ be a permutation $g\co \{1,\ldots, n\}\to \{1,\ldots, n\}$, $g_i=g(i)$.  We define 
\begin{equation}
Tcx(g)=\begin{cases}\prod_{r=0}^{n}\tau[0,\alpha_{r1},\alpha_{r2},\ldots, \alpha_{rn},r+1]_x & n\geq 1\\
\tau([0,1]_x)=\tau(x) & n=0
\end{cases} \label{def:Tcxg}
\end{equation}
where 
\begin{equation}
\alpha_{rj}=
\max(\{0,1,\ldots, r\}\cap\{0,g_1,g_2,\ldots, g_j\}), \quad j=1, \ldots,n. \label{def:aij}
\end{equation}
 When $n=0$, $g=()\in S_0=\emptyset$ is a formal notation. The most convenient way to identify  the  simplices in the image of the singular cube $|cx|$ is through the embedding $GX\to RGX$. Then $Tcx\in RGX$ is given by 
\[Tcx=\begin{cases}\sum_{g\in S_n}(-1)^{\mathrm{sign}(g)}Tcx(g) & n\geq 1\\
          \tau(x) & n=0.
           \end{cases}
\]
It is convenient to write the right-hand side of \eqref{def:Tcxg} in the form
\begin{equation}
  \tau\left[\begin{array}{cccccc}
    0 & \alpha_{01} & \alpha_{02} & \ldots & \alpha_{0n} & 1 \\ 
    0 & \alpha_{11} & \alpha_{12} & \ldots & \alpha_{1n} & 2 \\ 
     &  &  &\ldots  &   &   \\ 
    0& \alpha_{r1} & \alpha_{r2} & \ldots & \alpha_{rn} & r+1 \\ 
      &   &   & \ldots &   &   \\ 
    0 & \alpha_{n1} & \alpha_{n2} & \ldots & \alpha_{nn} & n+1 \\ 
  \end{array}\right]_x\label{def:Tcx1}
\end{equation} 
where each row corresponds to an element in $GX$. When the matrix \eqref{def:Tcx1} degenerates to $[0,1]_x$, it is understood that we are in the situation $n=0$. Here are some examples: when $n=1$, $S_1$ has a single element $g=(1)$, and 
\[Tcx(g)=
\tau\left[
\begin{array}{ccc}
  0&0   &1   \\
  0& 1  &2   
\end{array}
\right]_x=
\tau[0,0,1]_x\tau[0,1,2]_x;\] let $g=(2,3,1)\in S_3$, $n=3$, then 
\[
Tcx(g)=\tau 
\left[
\begin{array}{ccccc}
  0& 0  & 0 &0 &1 \\
  0& 0  & 0  &1 &2\\
  0& 2  & 2 & 2& 3\\
  0 & 2& 3&3 &4
\end{array}
\right]_x=
\tau[0,0,0,0,1]_x\tau [0,0,0,1,2]_x\tau [0,2,2,2,3]_x\tau[0,2,3,3,4]_x
\]
for $x=[0,1,2,3,4]_x$.

An advantage of the matrix form \eqref{def:Tcx1} is that we can use operations on columns, for instance, operations $s_iTcx(g)$ and $d_iTcx(g)$ (except $d_nTcx(g)$) can be obtained by doubling and deleting corresponding columns, respectively. In particular, we write 
\begin{equation}
 d_nTcx(g)=\tau \left[\begin{array}{cccccc}
    0 & \alpha_{01} & \alpha_{02} & \ldots & \alpha_{0n} &  \\ 
    0 & \alpha_{11} & \alpha_{12} & \ldots & \alpha_{1n} &  \\ 
     &  &  &\ldots  &   &   \\ 
    0& \alpha_{r1} & \alpha_{r2} & \ldots & \alpha_{rn} &  \\ 
      &   &   & \ldots &   &   \\ 
    0 & \alpha_{n1} & \alpha_{n2} & \ldots & \alpha_{nn} &  \\ 
  \end{array}\right]^{-1}_x
  \tau\left[\begin{array}{cccccc}
    0 & \alpha_{01} & \alpha_{02} & \ldots & \alpha_{0(n-1)} & 1 \\ 
    0 & \alpha_{11} & \alpha_{12} & \ldots & \alpha_{1(n-1)} & 2 \\ 
     &  &  &\ldots  &   &   \\ 
    0& \alpha_{r1} & \alpha_{r2} & \ldots & \alpha_{r(n-1)} & r+1 \\ 
      &   &   & \ldots &   &   \\ 
    0 & \alpha_{n1} & \alpha_{n2} & \ldots & \alpha_{n(n-1)} &  n+1\\ 
  \end{array}\right]_x
  \label{def:dnTcxg}
\end{equation}
where the multiplication is understood as that of the two elements in the same row. For example,
\begin{align*}
   & d_3\tau 
\left[
\begin{array}{ccccc}
  0& 0  & 0 &0 &1 \\
  0& 0  & 0  &1 &2\\
  0& 2  & 2 & 2& 3\\
  0 & 2& 3&3 &4
\end{array}
\right]_x
=\tau\left[
\begin{array}{cccc}
  0& 0  & 0 &0  \\
  0& 0  & 0  &1 \\
  0& 2  & 2 & 2\\
  0 & 2& 3&3 
\end{array}
\right]_x^{-1}\tau\left[
\begin{array}{cccc}
  0& 0  & 0 &1  \\
  0& 0  & 0  &2 \\
  0& 2  & 2 & 3\\
  0 & 2& 3&4 
\end{array}
\right]_x\\
&=(\tau[0000]_x)^{-1}\tau[0001]_x(\tau[0001]_x)^{-1}\tau[0002]_x(\tau[0222]_x)^{-1}\tau[0223]_x(\tau[0233]_x)^{-1}\tau[0234]_x\\
&=\tau[0002]_x\tau[0223]_x\tau[0234]_x
\end{align*}
after ignoring terms of the form $\tau(s_ny)$, $y\in X_n$.
\begin{lem}\label{lem:degeneracy}
Given $x\in X_{n+1}$ and $g\in S_{n}$, the element $Tcx(g)\in GX$ is degenerate if $x$ is. 
\end{lem}
\begin{proof}
Suppose that $x=s_ky$, namely $x=[0,1,\ldots, k-1, k,k,,k+1,\ldots, n]_y$. By definition, the matrix $Tc(s_ky)(g)$ is obtained from $Tcx(g)$, by setting $k=k+1$ in all entries whose values are $k,k+1$. The degeneracy of $Tcx(g)$ is equivalent to a coincidence of two adjacent columns.   

When $k\in\{1,\ldots,n-1\}$,  we have $g_i=k$ and $g_{i'}=k+1$, for some $i,i'\in\{1,\ldots, n\}$. Let $j_0=\max\{i,i'\}\geq 2$. Now
$\{0,g_1,\ldots,g_{j_0}\}=\{0,g_1,\ldots,g_{j_0-1}\}$
after setting $k=k+1$, hence the two columns $(\alpha_{rj_0})_{r=0}^n,(\alpha_{r(j_0-1)})_{r=0}^n$ coincide, by \eqref{def:aij}. 

When $k=0$, observe that $\{0,g_1,\ldots,g_i\}=\{0,g_1,\ldots, g_{i-1}\}$ where $g_i=1$. After setting $1=0$, the columns $(\alpha_{r(i-1)})_{r=0}^n, (\alpha_{ri})_{r=0}^n$ coincide (when $g_1=1$, it is easy to check that the second column coincides with the first one, since it contains only $0$ in all entries, after setting $1=0$).

Notice that we always have $\alpha_{nn}=n$ in the last row, by definition,  whose last entry is $n+1$. When $k=n$, after setting $n=n+1$ we delete the last row without changing $Tcx(g)$, by the relation $\tau(s_ny)=1$ for $y\in X_n$. Suppose $g_i=n$ for some $i\in\{1,\ldots,n\}$. We see that $\{0,g_1,\ldots,g_i\}$ and $\{0,g_1,\ldots, g_{i-1}\}$ coincide after intersecting with $\{0,1,\ldots, r\}$ whenever $r<n$, which means that 
the $i$-th and $i+1$-th columns coincide, after we delete the last row. 
\end{proof}
On the other hand, if $x$ is non-degenerate, so does $Tcx(g)$ with $g=(1,\ldots,n)$, since its last row gives $\tau x$.

\begin{lem} We have  
\begin{equation}
d_iTcx(gp_i)=d_iTcx(g), \quad i=1,\ldots, n-1, \label{eq:faces0}
\end{equation} 
where $g=(g_i)_{i=1}^n\in S_n$ is any given element and $p_i\in S_n$ is the generator such that
\[gp_i=(g_1, g_2, \ldots g_i, g_{i+1}, \ldots ,g_n)\circ p_i=(g_1, g_2,\ldots, g_{i+1},g_i, \ldots, g_n).\]
\end{lem}
\begin{proof}
By definition, the only difference between $Tcx(g)$ and $Tcx(gp_i)$ will be in the $i+1$-th column of the matrix \eqref{def:Tcx1}, which will be removed by $d_i$.
\end{proof}
It is easily checked that $p_i^2=\mathrm{id}_{S_n}$, and $\mathrm{sign}(g)$ and $\mathrm{sign}(gp_i)$ are different. The lemma above shows that $d_iTcx=\sum_{g\in S_n}(-1)^{\mathrm{sign}(g)}d_iTc(g)=0$ for $i=1,\ldots,n-1$, since the summands cancel each other in pairs. We have already proved the following.
\begin{cor}\label{cor:d0dn} In the simplicial $R$-module $RGX$ we have $dTcx= d_0Tcx+ (-1)^{n}d_{n}Tcx$, where $d=\sum_i (-1)^id_i$, $x\in X_{n+1}$.
\end{cor}

Let $y=[0,\ldots, l+1]_y=[y_0,y_1\ldots, y_{l+1}]_x$ be a (possibly degenerate) face of $x$ of dimension $l+1$, $y_0\leq y_1\leq\ldots\leq y_{l+1}$, and let $g\in S_{l}$. By definition
\[
Tcy(g)=\prod_{r=0}^{l}\tau([0,\alpha_{r1},\alpha_{r2},\ldots, \alpha_{rl},r+1]_y)=\prod_{r=0}^{l}\tau([y_0,y_{\alpha_{r1}},y_{\alpha_{r2}},\ldots, y_{\alpha_{rl}},y_{r+1}]_x),
\]
with $\alpha_{rj}=\max(\{0,1\ldots, r\}\cap\{0,g_1,\ldots, g_j\})$.
In particular, when $y=[0,1\ldots,k]_x$, $g'=(g'_i)_{i=1}^{k-1}$, we have 
\begin{equation}
 Tcy(g')= \tau\left[\begin{array}{cccccc}
    0 & \alpha_{01}'& \alpha_{02}' & \ldots & \alpha_{0(k-1)}' & 1 \\ 
    0& \alpha_{11}' & \alpha_{12}' & \ldots & \alpha_{1(k-1)}' &  2 \\ 
      &   &   & \ldots &   &   \\ 
    0 & \alpha_{(k-1)1}' & \alpha_{(n-k)2}' & \ldots & \alpha_{(k-1)(k-1)}' & k \\ 
  \end{array}\right]_x\label{def:Tcyg'}
\end{equation}
with $\alpha_{rj}'=\max(\{0,1\ldots, r\}\cap\{0,g_1',\ldots, g_j'\})$; when $y=[k,k+1,\ldots, n+1]_x$, $g''=(g_i'')_{i=1}^{n-k}\in S_{n-k}$, we have 
\begin{equation}
 Tcy(g)=\tau \left[\begin{array}{cccccc}
    k & k+\alpha_{01}'' & k+\alpha_{02}'' & \ldots & k+\alpha_{0(n-k)}'' & k+1 \\ 
    k& k+\alpha_{11}'' & k+\alpha_{12}'' & \ldots & k+\alpha_{1(n-k)}'' & k+2 \\ 
      &   &   & \ldots &   &   \\ 
    k & k+\alpha_{(n-k)1}'' & k+\alpha_{(n-k)2}'' & \ldots & k+\alpha_{(n-k)(n-k)}'' & n+1 \\ 
  \end{array}\right]_x\label{def:Tcyg''}
\end{equation}
where $\alpha_{rj}''=\max(\{0,1\ldots, r\}\cap\{0,g_1'',\ldots, g_j''\})$.

Let $g'\in S_{k-1}$ and $g''\in S_{n-k}$ be two elements, $w=\prod_{i=1}^{n-1}x_{n-i}\in\mathrm{Shuf}(a^{k-1},b^{n-k})$ a word. Using the notation \eqref{def:w1}, we identify the multiplication
\[ \left(s_{x_b^-(w)}Tc([0,k]_x)(g') \right) \left(s_{x_a^-(w)}Tc([k,n+1]_x)(g'')\right)\]
with 
 \begin{equation}
  \begin{array}{ccccccccc}
   & x_1 & x_2 & \ldots &x_i & x_{i+1} &\ldots & x_{n-1} &\\
  \hline
  &  0 & \beta_{02}' &\ldots &\beta_{0i}' & \beta_{0(i+1)}' &\ldots &\beta_{0(n-1)}'  & \beta_{0n}' \\ 
  & 0  &\beta_{12}'&\ldots &\beta_{1i}' & \beta_{1(i+1)}' &\ldots &\beta_{1(n-1)}' &  \beta_{1n}' \\ 
I)  &      &\ldots &  &  &\ldots & & &  \\
   &0 &\beta_{(k-1)2}'&\ldots & \beta_{(k-1)i}' & \beta_{(k-1)(i+1)}' &\ldots &\beta_{(k-1)(n-1)}' & \beta_{(k-1)n}' \\ 
 \hline
   & k &\beta''_{k2} &\ldots&\beta''_{ki} & \beta''_{k(i+1)} &\ldots &\beta''_{k(n-1)} & \beta''_{kn} \\ 
II) &        &  \ldots &  &  &\ldots & & &\\
  & k &\beta''_{n2} &\ldots& \beta''_{ni} & \beta''_{n(i+1)} &\ldots & \beta''_{n(n-1)} &  \beta''_{nn} \\ 
  \end{array}\label{m:Shuf1}
\end{equation}
whose columns are generated inductively by the rules: A) the columns under $x_1$ in parts I), II), respectively, coincide with the first columns of the matrices \eqref{def:Tcyg'} and \eqref{def:Tcyg''}; B) suppose that the columns under $x_i$ in parts I), II), respectively, coincides with the $n_i',n_{i}''$-th columns of matrices \eqref{def:Tcyg'} and \eqref{def:Tcyg''}, when $x_i=a$ we have the next column under $x_{i+1}$ in part I) (resp. in part II)) coincides with the $(n_i'+1)$-th column of \eqref{def:Tcyg'} (resp. remains the same), similarly when $x_i=b$, the next column in part I) remains the same while the one in part II) coincides with the $(n_i''+1)$-th one in  \eqref{def:Tcyg''}. 

\begin{prop}\label{prop:d_0}
In $RGX$ we have 
\begin{equation}
\sum_{g\in S_{k,n}}(-1)^{\mathrm{sign}(g)}d_0Tcx(g)= (-1)^{k-1}Tc([0,k]_x)\cdot Tc([k,n+1]_x),\label{eq:d0}
\end{equation}
here $S_{k,n}=\{g_1=k\mid g=(g_i)_{i=1}^n\in S_n\}$, $k=1,\ldots, n$. Consequently as $k$ runs through $1$ to $n$,
\[
d_0Tcx=\sum_{k=1}^n (-1)^{k-1}Tc([0,k]_x)\cdot Tc([k,n+1]_x).
\]
\end{prop}
\begin{proof}
Let
\[
    \phi_k: \mathrm{Shuf}(a^{k-1},b^{n-k})\times S_{k-1}\times S_{n-k}\to S_{k,n}
\]
be the map sending $(w,g',g'')$, $w=x_{n-1}\ldots x_1$, $g'=(g'_i)_{i=1}^{k-1}$ and $g''=(g''_i)_{i=1}^{n-k}$, to $g=(g_i)_{i=1}^n$ with $g_1=k$, and 
\begin{equation*}
g_{1+a_i}=g'_i, \quad g_{1+b_j}=g''_{j}+k 
\end{equation*}
 corresponding to 
\begin{equation*}
x_{a}(w)=x_{a_{k-1}}\ldots x_{a_1}, \ i=1,\ldots, k-1, \quad 
x_{b}(w)=x_{b_{n-k}}\ldots x_{b_1}, \ j=1,\ldots, n-k, 
\end{equation*}
respectively. In other words, we have the correspondence 
\begin{equation}
         k, \ \stackrel{x_1}{g_2}, \ \stackrel{x_2}{g_3}, \ \ldots, \stackrel{x_i}{g_{i+1}}, \ \ldots, \ \stackrel{x_{n-1}}{g_n}, \label{def:phi_k2}
\end{equation}
where for those entries where $x_i=a$ (resp. $x_i=b$) we fill $g_1',\ldots, g_{k-1}'$ in order (resp. $g''_1+k,\ldots, g''_{n-k}+k$ in order). To get the sign of $g$, we use permutations to change $g',g''$ into $(1,\ldots,k-1),(1,\ldots,n-k)$, respectively, and then put them in correct positions according to the permutation by which $w\to b^{n-k}a^{k-1}$, and finally we put the beginning $k$ to its correct place:
\begin{equation}
\mathrm{sign}(g)=k-1+\mathrm{sign}(w)+\mathrm{sign}(g')+\mathrm{sign}(g'').\label{eq:signg}
\end{equation}  
Clearly $\phi_k$ is injective, implying that it is an isomorphism since we have  $(n-1)!$ elements on both sides.

By definition (see \eqref{def:cdot})
\begin{equation}
Tc[0,k]_x\cdot Tc[k,n+1]_x=\sum_{g'\in S_{k-1},g''\in S_{n-k}, \atop 
w\in \mathrm{Shuf}(a^{k-1},b^{n-k})} (-1)^{\mathrm{sign}(w,g',g'')} \left(s_{x_b^-(w)}Tc[0,k]_x(g')\right) \left(s_{x_a^-(w)}Tc[k,n+1]_x(g'')\right),
\label{eq:d02}
\end{equation}
$\mathrm{sign}(w,g',g'')=\mathrm{sign}(w)+\mathrm{sign}(g')+\mathrm{sign}(g'')$. Together with \eqref{eq:signg}, it suffices to show 
\begin{equation}
\left(s_{x_b^-(w)}Tc[0,k]_x(g')\right) \left(s_{x_a^-(w)}Tc[k,n-k]_x(g'')\right)=d_0Tcx(g),\quad g=\phi_k(w,g',g''), \label{eq:d03}
\end{equation}
$w=\prod_{i=1}^{n-1}x_{n-i}$, which will be done by a comparison between \eqref{m:Shuf1} and the matrix $d_0Tcx(g)$, where the latter is obtained by deleting the first column of that of $Tcx(g)$:
\[
  \begin{array}{ccccccccc}
    & x_1& x_2 & \ldots &x_i & x_{i+1} &\ldots & x_{n-1} &\\
  \hline
    & 0 &\alpha_{02} &\ldots &\alpha_{0i} & \alpha_{0(i+1)} &\ldots &\alpha_{0n}  & 1 \\ 
     & 0&\alpha_{12}&\ldots &\alpha_{1i} & \alpha_{1(i+1)} &\ldots &\alpha_{1n} & 2 \\ 
I)       & &\ldots &  &  &\ldots & & &  \\
   & 0 &\alpha_{(k-1)2}&\ldots & \alpha_{(k-1)i} & \alpha_{(k-1)(i+1)} &\ldots &\alpha_{(k-1)n} & k \\ 
   \hline
   & k &\alpha_{k2}&\ldots&\alpha_{ki} & \alpha_{k(i+1)} &\ldots &\alpha_{kn} & k+1 \\ 
 II)      &  &  \ldots &  &  &\ldots & & &\\
    & k&\alpha_{n2}&\ldots& \alpha_{ni} & \alpha_{n(i+1)} &\ldots &\alpha_{nn} & n+1 \\ 
  \end{array}
\]
whose first column $(\alpha_{r1})_{r=0}^n$, $\alpha_{r1}=\max(\{0,\ldots, r\}\cap \{0,g_1\})$ is obtained from the assumption $g_1=k$. Suppose that the $i$-th column of the matrix above, $1\leq i\leq n-1$, already coincides with that of $\eqref{m:Shuf1}$, namely 
\[\alpha_{ri}=\begin{cases}\beta_{ri}'=\max(\{0,1,\ldots, r\}\cap \{0,g_1',\ldots, g_{n_i'-1}'\} & r\leq k-1\\
\beta_{ri}''=k+\max(\{0,1,\ldots, r-k\}\cap \{0,g_1'',\ldots, g_{n_i''-1}''\}  & r\geq k.
\end{cases}\] 
Now $\alpha_{r(i+1)}=\max(\{0,\ldots, r\}\cap \{0,g_1,\ldots, g_{i+1}\})$, by definition.
If $x_i=a$, by \eqref{def:phi_k2} we have $g_{i+1}=g'_{n_i'}<k$, $n_i'\in\{1,\ldots,k-1\}$, then
\[\alpha_{r(i+1)}=\begin{cases}\max(\{0,\ldots, r\}\cap \{0,g'_1,\ldots, g'_{n_i}\})=\beta_{r(i+1)}'  & r\leq k-1 \\ k+\max(\{0,1,\ldots, r-k\}\cap \{0,g_1'',\ldots, g_{n_i''-1}''\}=\beta_{ri}'' & r\geq k;\end{cases}\]
if $x_i=b$, we have $g_{i+1}=g''_{n''_i}+k$, $n''_i\in\{1,\ldots, n-k\}$, then
\[
\alpha_{r(i+1)}=\begin{cases}\max(\{0,\ldots, r\}\cap \{0,g'_1,\ldots, g'_{n_i-1}\})=\beta_{ri}'  & r\leq k-1 \\ k+\max(\{0,\ldots, r-k\}\cap \{0,g_1'',\ldots, g_{n_i''}''\})=\beta_{r(i+1)}'' & r\geq k.\end{cases}
\]
We see that the two matrices coincide, by Rule B) under \eqref{m:Shuf1}, hence \eqref{eq:d03} holds. Together with \eqref{eq:d02} and \eqref{eq:signg}, we get \eqref{eq:d0}. The second statement is clear.
\end{proof}

\begin{prop}\label{prop:d_n}
In $RGX$ we have \begin{equation}
(-1)^n\sum_{g\in S_{n,k}}(-1)^{\mathrm{sign}(g)}d_nTcx(g)=(-1)^{k}Tcd_kx=(-1)^{k}\sum_{g'\in S_{n-1}}(-1)^{\mathrm{sign}(g')}Tcd_kx(g'), \label{eq:d_n}
\end{equation}
 where $S_{n,k}=\{g=(g_i)_{i=1}^n\in S_n\mid g_n=k\}$. Consequently, as $k$ runs through $1$ to $n$,
\[
(-1)^n d_nTcx=\sum_{k=1}^n (-1)^kTcd_kx.
\]
\end{prop}
\begin{proof}
First we show that 
\begin{equation}
   d_nTcx(g)=   \tau \left[\begin{array}{cccccc}
    0 & \alpha_{01} & \alpha_{02} & \ldots & \alpha_{0(n-1)} & 1 \\ 
 &  &  &\ldots  &   &   \\ 
     0& \alpha_{(k-1)1} & \alpha_{(k-1)2} & \ldots & \alpha_{(k-1)(n-1)} & k-1 \\ 
    0& \alpha_{(k+1)1} & \alpha_{(k+1)2} & \ldots & \alpha_{(k+1)(n-1)} & k+1 \\ 
      &   &   & \ldots &   &   \\ 
    0 & \alpha_{n1} & \alpha_{n2} & \ldots & \alpha_{n(n-1)} & n+1 \\ 
  \end{array}\right]_x \label{eq:dnTcxgk}
\end{equation}
which is obtained from $Tcx(g)$ by deleting the $k$-th row and the $(n+1)$-th column. To see this, for $g=(g_i)_{i=1}^n$ with $g_n=k$, we have 
\[
 \alpha_{r(n-1)}=\max(\{0,\ldots,r\}\cap \{0,g_1,\ldots,g_{n-1}\})=
 \max(\{0,\ldots,r\}\cap \{0,g_1,\ldots,g_{n-1},k\})=\alpha_{rn}
\]
unless $r=k$, where $\alpha_{k(n-1)}=k-1$ and $\alpha_{kn}=k$, respectively. Using \eqref{def:dnTcxg} and the rule $\tau s_nz=1$ for $z\in X_n$, we see that the non-trivial part of the block on the left-hand side of \eqref{def:dnTcxg} concentrates in row $k+1$, which eliminates the $k$-th row in the block on the right-hand side (since the entries in these two rows are identical, they cancel each other after multiplication), whose remaining rows give \eqref{eq:dnTcxgk}. 
  
Next we define a map $\psi_k:S_{n,k}\to S_{n-1}$ by sending $g=(g_i)_{i=1}^n$ to $g'=(g'_i)$, where $g'_i=g_i$ if $g_i<k$ otherwise $g_i'=g_{i}-1$ if $g_i>k$, $1\leq i\leq n-1$. We claim that \eqref{eq:dnTcxgk} coincides with 
\begin{equation}
    Tcd_kx(\psi_k(g))=\tau 
 \left[\begin{array}{cccccc}
    0 & \alpha_{01}' & \alpha_{02}' & \ldots & \alpha_{0(n-1)}' & 1 \\ 
 &  &  &\ldots  &   &   \\ 
     0& \alpha_{(k-1)1}' & \alpha_{(k-1)2}' & \ldots & \alpha_{(k-1)(n-1)}' & k-1 \\ 
    0& \alpha_{k1}' & \alpha_{k2}' & \ldots & \alpha_{k(n-1)}' & k \\ 
      &   &   & \ldots &   &   \\ 
    0 & \alpha_{(n-1)1}' & \alpha_{(n-1)2}' & \ldots & \alpha_{(n-1)(n-1)}' & n\\ 
  \end{array}\right]_y\label{eq:Tdckx}
\end{equation}
where $y=d_kx$, namely $y=[y_0,\ldots, y_n]_x$ with $y_i=\begin{cases} i & i<k \\ i+1 & i\geq k\end{cases}$. To see this, notice that $\alpha_{rj}=\max(\{0,\ldots, r\}\cap\{0,g_1,\ldots, g_j\})\geq k+1$ if and only if both $r\geq k+1$ and $\max\{g_1,\ldots, g_j\}\geq k+1$. This implies that $\max\{g_1',\ldots, g_j'\}=\max\{g_1,\ldots, g_j\}-1$, by the construction of $\psi_k$, therefore 
\[
\alpha_{(r-1)j}'=\max(\{0,\ldots, r-1\}\cap\{0,g'_1,\ldots, g'_j\})=\alpha_{rj}-1,
\]
when $\alpha_{rj}\geq k+1$ (which implies that $r\geq k+1$). It is easy to see that $\alpha_{rj}'=\alpha_{rj}$ when $\alpha_{rj}\leq k$. We see that \eqref{eq:Tdckx} is obtained from \eqref{eq:dnTcxgk} by changing each entry $u$ to $u-1$ whenever $u\geq k+1$, i.e., the claim holds. Finally it can be easily  checked that $\psi_k$ is bijective, and 
\[\mathrm{sign}(g)=\mathrm{sign}(\psi_k(g))+n-k,
\]
which can be done by first moving $g_1,\ldots, g_{n-1}$ to the correct order by a permutation, namely the one $\psi_k(g)\to (1,\ldots,n-1)$ since $g_n=k$ is not involved, then we put the last $g_n=k$ to its correct position, giving the desired sign $(-1)^{n-k}$, thus \eqref{eq:d_n} follows. The second statement is clear.
\end{proof}

We consider $\square^n$ as a simplicial set, being the cartesian product of $n$ copies of $[0,1]$. A triangulation of $\square^n$ by $n!$ simplices is given by
\[
  \underbrace{[0,1]\times\ldots\times [0,1]}_{n}=\sum_{g\in S_n}(-1)^{\mathrm{sign}(g)}w_g([0,1],\ldots, [0,1]),
\]
where $w_g=a_{g_n}\ldots a_{g_1}$ for $g=(g_1,\ldots, g_n)\in S_n$. In other words, $w\in\mathrm{Shuf}(a_1,\ldots, a_n)$, $a_k$ the letter associated to the $k$-th copy of $[0,1]$, $k=1,\ldots,n$. Topologically $w_g([0,1],\ldots, [0,1])$ is the convex hull of ordered vertices 
\begin{equation}
v_0=(0,\ldots,0),v_1,\ldots, v_{n-1}, v_n=(1,\ldots,1), \quad v_i\in\{0,1\}^n\subset [0,1]^n,\label{def:vi}
\end{equation} 
where for $i=1,\ldots, n$, the coordinates of $v_{i-1}$ and $v_{i}$ coincide except at the $g_i$-th component, which is $0$ in $v_{i-1}$ and $1$ in $v_i$, coinciding with \eqref{def:w1} when $n=2$.

We define its $k$-th top face $d_k^0\square^n$ (resp. $k$-th bottom face $d_k^1\square^n$) as the subcomplex consisting of points whose $k$-th coordinate is $1$ (resp. whose $k$-th coordinate is $0$). The following lemma is well-known.
\begin{lem}\label{lem:wg}
The $k$-th top face $d_k^0\square^n$ of $\square^n=\bigcup_{g\in S_n}w_g([0,1],\ldots, [0,1])$, is the union $\bigcup_{g\in S_{k,n}}d_0w_{g}([0,1],\ldots,[0,1])$, $S_{k,n}=\{g\in S_n\mid g_1=k\}$, and its $k$-th bottom face $d_k^1\square^n$ is the union $\bigcup_{g\in S_{n,k}}d_nw_{g}([0,1]^n)$, $S_{n,k}=\{g\in S_n\mid g_n=k\}$. 
Moreover, as $n$-simplices in the cartesian product $[0,1]^n$, we have 
\begin{equation}
d_iw_g([0,1],\ldots,[0,1])=d_iw_{gp_i}([0,1],\ldots,[0,1]), \quad i=1,\ldots, n-1, \label{eq:faces1}
\end{equation} 
where $gp_i=(g_i')_{i=1}^n$ is obtained from $g=(g_i)_{i=1}^n$ by permuting $i, i+1$ entries, namely $g_i'=g_{i+1}$ and $g'_{i+1}=g_i$ and coincide in other entries.
\end{lem}
\begin{proof}
We see that the $k$-th top face of $\square^n$, consisting of points whose $k$-th coordinate is $1$, is the union of $(n-1)!$ simplices in which any of them is a convex hull of $v_1,\ldots, v_n$ such that the $k$-th coordinate of all $v_i$ is $1$. The cone $v_0*d_k^0\square^n$ of $v_0$ with these $(n-1)!$ simplices gives the union $\cup_{g\in S_{k,n}}w_{g}([0,1]^n)$, so $d_k^0\square^n=\cup_{g\in S_{k,n}}d_0w_{g}([0,1]^n)$.  Similarly $\cup_{g\in S_{n,k}}w_{g}([0,1]^n)$ is the cone $d_k^1\square^n*v_n$ of $d_k^1\square^n$, the union of all non-degenerate simplices of dimension $n-1$ whose $k$-th coordinate is $0$, with $v_n$. We see that $d_k^1\square^n=\cup_{g\in S_{n,k}}d_nw_{g}([0,1]^n)$.    

For the second statement, the sequence \eqref{def:vi} of vertices $v_0',\ldots, v_n'$ corresponding to $w_{gp_i}([0,1],\ldots,[0,1])$ only differs in $v_i'$, whose $g_i'$-component is $1$. We see that \eqref{eq:faces1} holds by deleting the $i$-th vertices in the two sequences.
\end{proof}

\begin{proof}[Proof of Statements a)-d) in Theorem \ref{thm:main0}]
For a), we associate a cube $cx=\square^n_x$ to every $x\in X_{n+1}$ , and for every given $g\in S_{n}$, we define the image of the $n$-simplex $w_g([0,1],\ldots,[0,1])$ in Lemma \ref{lem:wg} under the map $|cx|$ as the $n$-simplex $Tcx(g)$, preserving the order of vertices of the simplices. This definition is clearly functorial with respect to simplicial maps, which is d). The $n!$ pieces match well to give a map $|cx|\co \square^n_x\to |GX|$, by a comparison of \eqref{eq:faces0} and \eqref{eq:faces1}. Moreover, it is a morphism of topological monoids, by their definitions \eqref{def:monoid1}, \eqref{def:cx},  and the well-know fact on the triangulation of a cartesian product of simplices and their shuffle product. Together with Propositions \ref{prop:d_0}, \ref{prop:d_n}, by choosing $R=\mathbb{Z}$, we get c). The statement b) follows from Lemma \ref{lem:degeneracy}.
\end{proof}
It remains to prove e). Recall that a space is abelian if its fundamental group acts trivially on all homotopy groups. In particular, a topological monoid $M$ is abelian since for two maps $f_1\co I\to M$ and $f_n\co I^n\to M$ representing two elements in $\pi_1$ and $\pi_n$, respectively, where $I=[0,1]$ is the unit interval, their cartesian product 
\[I\times I^n\stackrel{(f_1,f_n)}{\longrightarrow} M\times M\stackrel{\cdot}{\longrightarrow} M\] 
gives a homotopy between $f_n$ which is represented by $\{0\}\times I^n\to M$, and action of $f_1$ on $f_n$ which is represented by $E^{n}\to M$, where $E^n$ is the union of other faces on the boundary of the $(n+1)$-cube $I\times I^n$. 
Together with the fact that 
a map between two connected abelian $CW$ complexes inducing isomorphisms on all homology groups is a homotopy equivalence (see for example, Hatcher \cite[Page 418, Prop. 4.74]{Hat02}),  the lemma below is clear.
\begin{lem}\label{lem:Hatcher}
    A morphism $f\co M\to M'$ of connected topological monoids $M,M'$ induces a weak homotopy equivalence if and only if it induces an isomorphism $H_*(M;\mathbb{Z})\cong H_*(M';\mathbb{Z})$ of homology groups. 
\end{lem}

\begin{prop}\label{prop:fil}
Let $X$ be the simplicial set obtained from the reduced simplicial set $X'$ by attaching a single non-degenerate simplex $x'$ of dimension $n+1$, $n\geq 0$.  
If $|T|\co |CX'|\to |GX'|$ is a homotopy equivalence, so does $|\widetilde{T}|\co|\widetilde{C}X|\to |GX|$, where 
\[\widetilde{C}X=\begin{cases}CX & n\geq 1\\
CX_{cx} & n=0.
\end{cases}
\] 
Here $CX_{cx}$ is the localization of the free monoid $CX$ with respect to $cx$. In other words, $CX_{cx}$ is obtained from the free monoid $CX$ by adding one more generator $y'$ of degree $0$, as well as relations $cx\cdot y'=y'\cdot cx=1$. 
\end{prop}
\begin{proof}
Let $S^{n+1}_{x}$ be the simplicial sphere in which $x$ of dimension $n+1$ is the only non-degenerate simplex (besides the base point), such that $d_ix$ is a degeneracy of the base point for all $i$. As discrete groups we have an isomorphism $G_kX=G_kX'*G_kS^{n+1}_{x}$ as a free product in each dimension $k\geq 0$. 

For a given discrete group $G$, let $S(G)\subset G$ be the set collecting all elements except the identity. 
By the structure theorem of amalagams, an element in $G_kX'*G_kS^{n+1}_{x}$ other than the identity is uniquely presented in one of the forms below:
\begin{equation}
   g_1'g_1g_2'g_2\cdots g_{l}'g_l, \ g_1g_1'g_2g_2'\cdots g_lg_l', \ g_1'g_1g_2'g_2\cdots g_{l}g_{l+1}', \ g_1g_1'g_2g_2'\cdots g_{l-1}'g_{l}, \label{bas:RGX'}
\end{equation}
with $g_i',g_i$ running through elements from $S(G_kX')$ and $S(G_kS^{n+1}_{x})$, respectively. For $i=1,2,3,4$ and $l\geq 1$, let $C_{k,l}^i$ be the collection of the four types of elements in \eqref{bas:RGX'}, respectively, where $C_{k,1}^4=S(G_kS^{n+1}_{x})$. Formally we define $C_{k,0}^3=S(G_kX')$. 

Now we filtrate the simplicial set $GX$ such that $F_{-1}GX=\cup_{k=0}^\infty\{\mathrm{id}_{G_kX}\}$, \[F_0GX=\bigcup_{k=0}^{\infty}C_{k,0}^3\cup \{\mathrm{id}_{G_kX}\}=GX',
\quad F_t GX= \bigcup_{l\leq t}\bigcup_{k=0}^{\infty}\bigcup_{i=1}^4 C_{k,l}^i\cup\{\mathrm{id}_{G_kX}\}\] for $t\geq 1$.  It is easily checked that $F_tGX$ is a simplicial subset of $GX$, moreover, the quotient $F_tGX/F_{t-1}GX$ for each $t\geq 0$, as a simplicial set, is the wedge of the following four simplicial sets (here the smash product $Z\wedge Z'$ is the quotient of the cartesian product $Z\times Z'$ by the subspace $*\times Z'\cup Z\times *'$,  $*\in Z,*'\in Z'$ the base points, respectively).
\begin{equation}
   Z'\wedge Z \wedge\ldots\wedge Z'\wedge Z, \ Z\wedge Z' \wedge\ldots\wedge Z\wedge Z', \ Z'\wedge Z\wedge Z' \wedge\ldots\wedge Z\wedge Z', \ Z\wedge Z'\wedge Z \wedge\ldots\wedge Z'\wedge Z, \label{sma:ZZ'}
\end{equation}
each containing exactly $t$ copies of $Z$, where $Z'=GX'$ and $Z=GS^{n+1}_{x}$.

For a given simplicial set $Y$ which is reduced with base point $*$, let $S(CY)$ be the set $CY\setminus \{1\}$, with $1=c[0,0]_*$. Notice that we have an isomorphism $CX=CX'*CS^{n+1}_{x}$, a free product of monoids.  Therefore, by an argument similar to that for free products of groups, every element of $CX$ except the identity admits a unique presentation in one of the forms  of \eqref{bas:RGX'},
with $g_i',g_i$ running through $S(CX'), S(CS^{n+1}_{x})$, respectively. 

For $i=1,2,3,4$, let $B_l^i$ be the collection of all basis elements of the $4$ types as in \eqref{bas:RGX'}, respectively, in which $l\geq 0$ in $B_l^3$ with $B_0^3=S(CX')$, and $l\geq 1$ in $B_l^i$ for $i=1,2,4$, with $B_1^4=S(CS^{n+1}_{x})$.

 We endow the cubical complex $|CX|$ with a filtration such that $F_{-1}(|CX|)=|\{1\}|$, the base point, $F_0(|CX|)=|CX'|=|B_0^3\cup\{1\}|$, and $F_t(|CX|)=\cup_{l\leq t} \cup_{i=1}^4|B_l^i\cup\{1\}|$, for all $t\geq 1$, which is well-defined since
\[d_i^{0,1}F_t(CX)\subset F_t(CX)\] 
coming from the observation that for $Y=X'$ or $Y=S_x^{n+1}$, $c[0,\ldots,i]_y\cdot c[i,\ldots, n+1]_y$ and $cd_iy$ remain in $CY$, if $y\in Y$. 
 Clearly $|CX|=\cup_{t\geq -1}F_t(|CX|)$.  Again by checking \eqref{bas:RGX'} we see that the associated quotient $F_t(|CX|)/F_{t-1}(|CX|)$, is homotopy equivalent to the wedge of smashes of the form \eqref{sma:ZZ'}, where $Z'=|CX'|$ and $Z=|CS^{n+1}_{x}|$.         

We define $F_t|GX|=|F_t(GX)|$, $t\geq 0$. By a), c) in Theorem \ref{thm:main0}, $|T|\co |CX|\to |GX|$ is filtration preserving. More precisely, $T$ maps an elements in $CX$ of the form \eqref{bas:RGX'} to an element in $GX$ of the same form, hence $|T|$ sends the geometric realization of a smash product of the form $\eqref{sma:ZZ'}$ to that of the same form.  To prove the weak homotopy equivalence, by Lemma \ref{lem:Hatcher}, it suffices to show that $T$ induces a morphism of spectral sequences with respect to the filtrations above, whose $E^1$-pages are isomorphic to the homology of the smash products of the form $\eqref{sma:ZZ'}$, respectively. After a comparison of the $E^1$-pages we see that it suffices to prove the isomorphism $H_*(|CS^{n+1}_{x}|;\mathbb{Z})\to H_*(|GS^{n+1}_{x}|;\mathbb{Z})$, by the K\"{u}nneth formula. It is well known that both of them are isomorphic to $\mathbb{Z}[u]$ when $n\geq 1$, the polynomial algebra with $u$ of degree $n$, corresponding to the classes represented $cx\in CS^{n+1}_{x}$ and $\tau x\in GS^{n+1}_{x}$, respectively, which are connected by $T$.   
When $n=0$ we have a commutative diagram 
\[
 \begin{CD}
               |CS^1_x| @>|T|>> |GS^1_x|\\
                @A\cong AA                @A\cong AA\\
                \mathbb{Z}_{\geq 0} @>\subset>> \mathbb{Z}
 \end{CD}
\]
with $\mathbb{Z}_{\geq 0}$ (resp. $\mathbb{Z}$) the set of non-negative integers as a topological monoid (resp. the set of integers as a topological group), endowed with the discrete topology, where the vertical maps are isomorphisms. After the localization we get the desired isomorphism.
\end{proof}

\begin{proof}[Proof of e) in Theorem \ref{thm:main0}]
The statement that $|T|:|CX|\to |GX|$ induces a morphism of topological monoids is already proved in d). By Lemma \ref{lem:Hatcher}, we only need to prove the isomorphism 
\begin{equation}
H_*(|CX|;\mathbb{Z})\stackrel{\cong}{\longrightarrow} H_*(|GX|;\mathbb{Z})\label{iso:hom}
\end{equation} 
induced by $|T|$. Let $\widetilde{C}X=S^{-1}CX$ be the localization of $CX$ with respect to the set $S=\{cx\mid x\in X_1\}$, where $X_1$ is the set of $1$-simplices of $X$ which contains the degeneracy $[0,0]_*$ of the base point $*$, and we have $c[0,0]_*=1$, the unit of $CX$.  The localization is well-defined since the monoid $CX$ is free, and we have a canonical inclusion $|CX|\to |\widetilde{C}X|$ of topological monoids,
such that the diagram
\[
    \begin{tikzcd}
{|CX|} \arrow[rr, "\subset"]\arrow[rrd, "{|T|}"'] & & {|\widetilde{C}X|} \arrow[d, "{|\widetilde{T}|}"]\\
                                    &  & {|GX|}
\end{tikzcd}
\]
commutes.

First we claim that $|\widetilde{T}|$ induces an isomorphism $H_*(|\widetilde{C}X|;\mathbb{Z})\longrightarrow H_*(|GX|;\mathbb{Z})$ of homology groups.  The special case when $X$ is simplicial set with a finite number of non-degenerate simplices can be done by an induction using Propositon \ref{prop:fil}. For a general $X$, every given cycle or boundary in the homology of $|\widetilde{C}X|$ or $|GX|$ has only finitely many simplices involved, therefore an element in the kernel or cokernel of the morphism \eqref{iso:hom} of homology groups is supported by a compact subset of $|X|$, which is reduced the a special case, hence the claim holds.

Next we consider the morphism $\omega\co H_*(|CX|;\mathbb{Z})\longrightarrow H_*(|\widetilde{C}X|;\mathbb{Z})$ induced by the inclusion. The assumption that $|X|$ is simply connected,  implies that $|GX|$ is connected by the long exact sequence of homotopy groups, hence $-cx$ and $(cx)^{-1}$ represent the same class in $H_0(|\widetilde{C}X|;\mathbb{Z})$, since their images coincide in $H_0(|GX|;\mathbb{Z})$ under $|\widetilde{T}|$ by the claim above. We see that every cycle in the chain complex $(\mathbb{Z}\widetilde{C}X,d)$ is represented by a finite sum in which each summand is a monomial without elements of the form $(cx)^{-1}$, since each of them can be replaced by $-cx$, up to a boundary, which means that $\omega$ is surjective. The injectiveness of $\omega$ comes from the observation that a cycle in $(\mathbb{Z}{C}X,d)$ bounds in $(\mathbb{Z}\widetilde{C}X,d)$ under the inclusion $CX\to \widetilde{C}X$ also bounds by the same element in $\mathbb{Z}{C}X$, after ignoring all summands with $(cx)^{-1}$ involved, as $d(cx)^{-1}=0$ these summands make no contributions after $d$. 

The two isomorphisms above give the desired isomorphism \eqref{iso:hom}.
\end{proof}
We state a byproduct as the corollary below, which is already proved by the arguments above.
\begin{cor}
    Let $X$ be a reduced simplicial set. The morphism $|\wt{T}|\co |\wt{C}X|\to |GX|$ of topological monoids induced from $|T|\co |CX|\to |GX|$ is a weak homotopy equivalence, where $\wt{C}X=S^{-1}CX$ is the localization of $CX$ with respect to the set $S=\{cx\mid x\in X_1\}$.

     Let $(R\wt{C}X,d)$ be the differential graded algebra generated by $cx$ of degree $n$, as $x=[0,\ldots,n+1]_x$ runs through non-degenerate simplices of dimension $n\geq 0$, as well as generators $(cx)^{-1}$ with relations $cx\cdot (cx)^{-1}=(cx)^{-1}\cdot cx=1$ when $n=0$, whose differential satisfies ($dcx=d(cx)^{-1}=0$ if $n=0$)
\begin{align*}
dcx=\sum_{i=1}^{n}(-1)^{i}(d_i^0cx-d_i^{1}cx)=\sum_{i=1}^{n}(-1)^i(c[0,\ldots,i]_x \cdot c[i,\ldots,n+1]_x-cd_ix),
\end{align*}
where $cd_ix$ (resp. $c[0,\ldots,i]_x$ or $c[i,\ldots,n+1]_x$) vanishes whenever $d_ix$ (resp. $[0,\ldots,i]_x$ or $[i,\ldots,n(x)+1]_x$) is degenerate. We have an isomorphism 
\[
     H_*(\Omega|X|;R)=H_*(R\wt{C}X,d)
\]
of graded algebras.
\end{cor}
The second part of the corollary above has been proved by Hess and Tonks \cite{HT10}. We may use a basis change for $R\wt{C}X$ so that its differential coincides with \eqref{eq:Ad2}.

\end{document}